\documentclass{amsart}
\usepackage[utf8]{inputenc}
\usepackage{color}

\usepackage[margin=1.0in]{geometry}

\usepackage{amsmath}
\usepackage{amsthm}
\usepackage{amssymb}
\usepackage[all]{xy}
\usepackage{enumitem}
\usepackage{graphicx}
\usepackage{mdwlist}
\usepackage{mathrsfs}\usepackage{longtable}
\usepackage{tikz}
\usetikzlibrary{patterns}
\usepackage{quiver}
\usepackage{caption}
\usepackage{subcaption}
\usepackage{makecell}
\usepackage{ textcomp }
\usepackage{longtable}

\newcommand\setItemnumber[1]{\setcounter{enum\romannumeral\@enumdepth}{\numexpr#1-1\relax}}

\newcommand{\p}{\mathbb{P}}
\DeclareMathOperator{\Bir}{Bir}
\DeclareMathOperator{\Aut}{Aut}

\DeclareMathOperator{\PGL}{PGL}

\DeclareMathOperator{\SL}{SL}

\DeclareMathOperator{\ZZ}{\mathbb{Z}}
\DeclareMathOperator{\QQ}{\mathbb{Q}}
\DeclareMathOperator{\RR}{\mathbb{R}}

\DeclareMathOperator{\KK}{K}

\DeclareMathOperator{\NS}{NS}

\DeclareMathOperator{\Sing}{Sing}

\DeclareMathOperator{\Exc}{Exc}

\DeclareMathOperator{\Ind}{Ind}
\DeclareMathOperator{\Ker}{Ker}

\DeclareMathOperator{\Reg}{Reg}
\DeclareMathOperator{\PsReg}{Ps-reg}
\DeclareMathOperator{\prReg}{pr.Reg}
\DeclareMathOperator{\prPsReg}{pr.Ps-reg}
\DeclareMathOperator{\AS}{AS}

\newcommand{\xdashrightarrow}[2][]{\ext@arrow 0359\rightarrowfill@@{#1}{#2}}

\newtheorem{theorem}[equation]{Theorem}
\newtheorem{lemma}[equation]{Lemma}

\newtheorem{corollary}[equation]{Corollary}

\theoremstyle{definition}
\newtheorem{defin}[equation]{Definition}

\newtheorem{example}{Example}

\newtheorem{question}[equation]{Question}

\makeatletter
\@addtoreset{equation}{section}
\makeatother

 \title{Regularization and algebraic stability of birational automorphisms}
 \author{Alexandra Kuznetsova}

\address{Higher School of Modern Mathematics in Moscow Institute of Physics and Technology, Dolgoprudny, Russia; Laboratory of Algebraic Geometry in Higher School of Economics, Moscow, Russia;  Steklov Mathematical Institute of Russian Academy of Sciences, Moscow, Russia.}

 \email{sasha.kuznetsova.57@gmail.com}

\begin{document}
\maketitle

\begin{abstract}
We compare five natural classes of birational automorphisms of an algebraic variety: projectively regularizable, regularizable, projectively pseudo-regularizable, pseudo-regularizable, and those admitting an algebraically stable model. These classes provide different geometric frameworks in which the dynamics of a birational automorphism, and in particular its first dynamical degree, can be studied. We determine the general inclusion relations among them and show that, for $\p^k$ with $k\geqslant 3$, all the resulting non-trivial inclusions are strict; we prove this by giving explicit examples.
\end{abstract}

\section{Introduction}
Given a birational automorphism $f$ of an algebraic variety $X$, we are usually interested in finding a birational model of $X$ on which one can easily study the geometric properties of $f$. For instance, if $f$ is a birational automorphism of finite order, then by \cite[Theorem 0.1]{Equivariant_resolution}, there exists a smooth projective birational model $\widetilde{X}$ of $X$ on which $f$ induces a regular automorphism. Thus, we can reduce the study of finite-order birational automorphisms of $X$ to that of regular automorphisms of its models. This reduction is highly useful for the classification of finite subgroups of birational automorphism groups; indeed, this method has been used in \cite{Bayle_Beauville}, \cite{Dolgachev_Iskovskikh}, and many other works.

If $f$ has infinite order and non-trivial dynamics, it is considerably more difficult to find a sufficiently nice birational model of $X$ that clarifies the properties of $f$. In particular, by \cite{CDX}, a very general birational automorphism of $\p^k$ for $k\geqslant 2$ does not admit a projective model on which $f$ induces a regular automorphism. On the other hand, an infinite-order birational automorphism can exhibit interesting dynamics on the variety $X$. One of the most important dynamical characteristics of $f\in\Bir(X)$ is its first dynamical degree (see Section \ref{section: preliminaries} for the definition). The first dynamical degree $\lambda_1(f)$ is a positive real number and a birational invariant of the pair $(X,f)$ that describes the asymptotic behavior of the inverse image operators~\mbox{$(f^n)^*\colon \NS(X)\to \NS(X)$}, where $\NS(X)$ is the Neron--Severi group of $X$. The properties of $\lambda_1(f)$ are deeply connected to the geometry and dynamics of $f$. For instance, if $\dim(X) = 2$, then $\lambda_1(f)$ is necessarily an algebraic integer, and in many cases, the properties of its minimal polynomial determine whether~$f$ induces a regular automorphism on a projective birational model of $X$ (see \cite{Blanc-Cantat}).

Therefore, we are interested in finding birational models of $X$ on which one can explicitly estimate or compute $\lambda_1(f)$. Let us introduce several classes of well-behaved birational models for $f$ that arise naturally in the study of its dynamical degree.
\begin{defin}\label{definition: main}
Let $X$ be a smooth algebraic variety over an algebraically closed field, and let $f\in \Bir(X)$ be a birational automorphism of $X$.
\begin{enumerate}
 \item We call $f$ a {\bf pseudo-automorphism} if $X$ contains Zariski open subsets $U$ and $V$ whose complements have codimension at least $2$ in $X$, such that $f$ induces an isomorphism $U \to V$.
 \item We say that $f$ is {\bf algebraically stable} if $X$ is projective and $(f^n)^* = (f^*)^n$ as operators on the Neron--Severi group $\NS(X)$ for all $n\geqslant 1$.
 \item  We say that $f$ is {\bf regularizable} (resp.\ {\bf pseudo-regularizable}) if there exists a smooth birational model $Y$ of $X$ such that $f$ induces a regular automorphism (resp.\ a pseudo-automorphism) of $Y$. We denote the set of regularizable (resp.\ pseudo-regularizable) birational automorphisms of $X$ by $\Reg(X)$ (resp.\ $\PsReg(X)$).
 \item We say that $f$ is {\bf projectively regularizable} (resp.\ {\bf projectively pseudo-regularizable}) if there exists a smooth projective birational model $Y$ of $X$ such that $f$ induces a regular automorphism (resp.\ a pseudo-automorphism) of $Y$. We denote the corresponding sets of birational automorphisms of $X$ by $\prReg(X)$ and $\prPsReg(X)$, respectively.
 \item We say that $f$ {\bf admits an algebraically stable model} if there exists a smooth projective birational model $Y$ of $X$ such that $f$ induces an algebraically stable birational automorphism of $Y$. We denote the set of birational automorphisms of $X$ admitting an algebraically stable model by $\AS(X)$.
\end{enumerate}
\end{defin}
The strongest of these conditions is regularizability especially if the regularization model $Y$ is a projective variety, in which case $f\in \prReg(X)$.
In this setting, the dynamics of $f$ can be studied using the geometry of $Y$. However, there are instances where $f$ is not regularizable, even on a quasi-projective model of $X$. A highly useful criterion for regularizable birational automorphisms was introduced in \cite{Lonjou_Urech_CAT0}. In that paper, the authors associate special $\mathrm{CAT}(0)$ cube complexes to an algebraic variety $X$, whose vertices are birational models of~$X$. Special subgroups of $\Bir(X)$ act on these complexes, and $f$ is regularizable if each complex contains an $f$-fixed vertex.

The idea of studying birational automorphisms that admit an algebraically stable model was suggested in \cite{Diller_Favre}. This class of birational automorphisms is particularly interesting because it allows one to readily compute their dynamical degrees. Namely, if $f$ is an algebraically stable birational automorphism of a smooth projective variety $X$, then~\mbox{$\lambda_1(f)$} equals the spectral radius of the pullback operator $f^*\colon \NS(X)\to \NS(X)$. In particular, if $f\in \AS(X)$, then $\lambda_1(f)$ is an algebraic number.

Projectively pseudo-regularizable birational automorphisms form a particularly useful subclass of $\AS(X)$. Such automorphisms also frequently arise in the study of birational maps with non-trivial dynamical degrees (see, for instance, \cite{Blanc_Pseudo, Bedford_Kim_Pseudo, BCK}).

Finally, the class $\PsReg(X)$ is of interest to us because, like $\Reg(X)$, it can be detected using the $\mathrm{CAT}(0)$ cube complex method. It serves as a natural generalization of the classes $\Reg(X)$ and $\prPsReg(X)$; however, relatively little is currently known about this class.

\medskip

The main goal of this paper is to compare the five classes of birational automorphisms $\prReg(X)$, $\Reg(X)$, $\prPsReg(X)$, $\PsReg(X)$, and $\AS(X)$, and to discuss the geometric or dynamical advantages of an element~$f$ lying in each of these classes.

Let us first examine the situation in low dimensions. If $X$ is a curve, then all its birational automorphisms are projectively regularizable; hence all five classes coincide.
If $X$ is a surface, then any $f\in \Bir(X)$ admits an algebraically stable model by \cite{Diller_Favre}; moreover, a pseudo-automorphism is necessarily regular. That is, we have the following stratification of the group of birational automorphisms of a surface $X$:
\[
\prReg(X)  = \prPsReg(X)\subset \Reg(X) =\PsReg(X) \subset \AS(X) = \Bir(X).
\]
If $X$ is a rational surface, all inclusions in this diagram are strict. Examples \ref{example 1} and \ref{example 2} below provide birational automorphisms of $\p^2$ lying in the complements $\Reg(\p^2)\setminus \prReg(\p^2)$ and $\Bir(\p^2)\setminus \Reg(\p^2)$, respectively.

If $\dim X \geqslant 3$, the situation is much more complicated. First, for higher-dimensional varieties, there is a difference between pseudo-regularizable and regularizable birational automorphisms. Moreover, by~\cite{BDJK}, there exists a birational automorphism $f$ of $\p^3$ whose first dynamical degree is a transcendental number, which implies that $f$ cannot be algebraically stable. The relations between these classes of birational automorphisms can be described by the following theorem.
\begin{theorem}
 Let $X$ be an algebraic variety over an uncountable field $K$ of characteristic zero. Then the classes $\prReg(X)$, $\Reg(X)$, $\prPsReg(X)$, $\PsReg(X)$ and $\AS(X)$ of birational automorphisms of $X$ fit into the following diagram.
 \begin{equation}\label{diagram}
 \xymatrix@R=0.5em@C=1em{
 \prReg(X) \, \ar@{}[r]|-*[@]{\subset}^*+{\text{\tiny ex. \ref{example 3}}}\, \ar@{}[dd]|-*[@]{\subset}^*+{\text{\tiny ex. \ref{example 1}}} & \prPsReg(X) \, \ar@{}[r]|-*[@]{\subset}^*+{\text{\tiny \hspace{2em}ex. \ref{example 2}}} \, \ar@{}[dd]|-*[@]{\subset}^*+{\text{\tiny ex. \ref{example 4}}}& \AS(X)\, \ar@{}[rd]|-*[@]{\hspace{-1.5em}\subset}^*+{\text{\tiny \hspace{-0.5em}ex. \ref{example 4}}} && \\
 &&& \AS(X) \cup \PsReg(X) \, \ar@{}[r]|-*[@]{\subset}^*+{\text{\tiny \hspace{3em} ex. \ref{example 5}}} & \Bir(X)\\
 \Reg(X)  \, \ar@{}[r]|-*[@]{\subset}^*+{\text{\tiny ex. \ref{example 3}}} & \PsReg(X) \, \ar@{}[rru]|-*[@]{\subset}_*+{\text{\tiny ex. \ref{example 2}}} &&&
 }
\end{equation}
If $X = \p^k$, where $k\geqslant 3$, then all inclusions in the diagram \textup{\eqref{diagram}} are strict, and Examples \textup{1--5} provide explicit birational automorphisms lying in the respective complements of each inclusion. The label above each inclusion refers to the number of the corresponding example.
\end{theorem}
While establishing that these classes fit into the diagram is straightforward, the primary contribution of the theorem is the strictness of all inclusions when $X$ is a projective space. Below, we provide the examples that prove this strictness.

\begin{example}\label{example 1}
Consider the algebraic torus $\mathbb{G}_m^2$ with coordinates $(x,y)$. Let $f\in \Aut(\mathbb{G}_m^2)$ be the regular automorphism given by $f(x,y) = (x,xy)$. It is easy to see that $f^n(x,y) = (x,x^n y)$. Thus, $f$ induces a birational automorphism of $\p^2$ whose degree growth is linear. Therefore, by \cite[Theorem 0.2]{Diller_Favre}, the birational automorphism of $\p^2$ induced by $f$ is not projectively regularizable. By Lemmas \ref{lemma: if f lies in S} and \ref{lemma: if f does not lie in S}, the birational automorphism $f\times \mathrm{id}_{\p^{k-2}}\in \Bir(\p^{k})$ remains regularizable but not projectively regularizable for all $k\geqslant 2$.
\end{example}
\begin{example}\label{example 2}
The following example was originally introduced in \cite{CDX}. Let $f\in \Bir(\p^k)$ be a birational automorphism of degree $d\geqslant 2$. If $A\in \PGL(k+1,\KK)$ is a very general element and $k\geqslant 2$, then the composition $g_A = A\circ f \in \Bir(\p^k)$ is not pseudo-regularizable by Theorem \ref{theorem: CDX is non-regularizable}.
\end{example}
\begin{example}\label{example 3}
Let $Q\subset \p^3$ be a smooth cubic surface. To each point $p\in Q$, one can associate a birational involution $\sigma_p$ of $\p^3$ (see Section \ref{section: pr ps-reg non-reg}). For points $p_1, p_2, p_3 \in Q$, the composition $f = \sigma_{p_3}\circ \sigma_{p_2}\circ\sigma_{p_1}$ is projectively pseudo-regularizable by \cite[Theorem 1.2]{Blanc_Pseudo}. However, if $Q$ is a very general cubic surface and if $p_1,p_2,p_3$ are sufficiently general points on $Q$, then $f$ is not regularizable by Theorem \ref{theorem: pr ps-reg non-reg}. By Lemmas \ref{lemma: if f lies in S} and \ref{lemma: if f does not lie in S}, the product $f\times \mathrm{id}_{\p^{k-3}}\in \Bir(\p^k)$ remains projectively pseudo-regularizable but not regularizable for all $k\geqslant 3$.
\end{example}
\begin{example}\label{example 4}
The group $\SL(3,\ZZ)$ acts regularly on the $3$-dimensional algebraic torus $\mathbb{G}_m^3$; that is, to each matrix $A\in \SL(3,\ZZ)$ one can associate a monomial automorphism $f_A$. Suppose $A\in \SL(3,\ZZ)$ is a matrix whose characteristic polynomial is irreducible over $\QQ$, with a real root $\beta$ and a pair of complex conjugate roots $\alpha, \overline{\alpha}$ such that $|\beta|<1<|\alpha|$. By Theorem \ref{theorem: ps-reg and non-AS}, the associated automorphism $f_A \in \Aut(\mathbb{G}_m^3)$ induces a regularizable birational automorphism of $\p^3$ with no algebraically stable model. Moreover, by Corollary \ref{corollary:  ps-reg and non-AS in any dimension}, for any $k\geqslant 3$ there exists $h\in \Aut(\mathbb{G}_m^{k-3})$ such that $f_A\times h\in\Aut(\mathbb{G}_m^k)$ admits no algebraically stable model.
\end{example}
\begin{example}\label{example 5}
The first dynamical degree of a birational automorphism can be a transcendental number; an example of such a map $f\in\Bir(\p^3)$ was constructed in \cite{BDJK}. By Lemma \ref{lemma: AS implies lambda is algebraic}, this birational automorphism does not admit an algebraically stable model. Furthermore, Theorem \ref{theorem: non-ps-reg and non AS} implies that $f$ is not pseudo-regularizable. Moreover, since $\lambda_1(f\times \mathrm{id}_{\p^{k-3}}) = \lambda_1(f)$, we deduce from Lemmas \ref{lemma: AS implies lambda is algebraic} and \ref{lemma: if f does not lie in S} that for all $k\geqslant 3$, the product birational automorphism $f\times \mathrm{id}_{\p^{k-3}}\in \Bir(\p^k)$ neither admits an algebraically stable model nor is pseudo-regularizable.
\end{example}
As explained above, all the classes of birational automorphisms discussed here arise naturally in the study of dynamical degrees. An important question is under what conditions the first dynamical degree $\lambda_1(f)$ is an algebraic number. As previously mentioned, this holds whenever $f$ admits an algebraically stable model. A very interesting recent result \cite[Theorem C]{AX} establishes that the first dynamical degree of a regular automorphism $f$ of an affine variety is an algebraic integer provided that $\lambda_1(f)^2>\lambda_2(f)$. Moreover, according to Theorem \ref{theorem: non-ps-reg and non AS}, the only known example of a birational automorphism with a transcendental dynamical degree is not pseudo-regularizable. It is therefore natural to ask the following question.

\begin{question}
  Let $f\in \Bir(X)$ be regularizable or pseudo-regularizable. Can $\lambda_1(f)$ be a transcendental number?
\end{question}
In Examples \ref{example 1}, \ref{example 3}, and \ref{example 5}, we first construct an automorphism $f\in \Bir(X)$ with the desired properties on~$X = \p^2$ or $\p^3$. To generalize these examples to arbitrary higher dimensions, we investigate the product map $f\times \mathrm{id}_{Y}$ on the variety $X\times Y$ and show that it shares the same properties as $f$. Indeed, it is easy to see that if $f$ belongs to one of the classes $\prReg(X)$, $\Reg(X)$, $\prPsReg(X)$, $\PsReg(X)$, or~$\AS(X)$, then~$f\times \mathrm{id}_Y$ belongs to the same class (see Lemma \ref{lemma: if f lies in S}). A more challenging question is whether $f\times \mathrm{id}_Y$ inherits the property of \emph{not} belonging to one of these classes. In Lemma \ref{lemma: if f does not lie in S}, we show that if $f$ does not lie in the classes~$\prReg$, $\Reg$, or $\PsReg$, then neither does $f\times \mathrm{id}_Y$. However, the following question remains open.
\begin{question}
 Let $X$ and $Y$ be algebraic varieties and let $f\in \Bir(X)$ be a birational automorphism of~$X$. Assume that $f\times \mathrm{id}_Y\in \Bir(X\times Y)$ is projectively pseudo-regularizable or admits an algebraically stable model. Is $f$ itself  necessarily projectively pseudo-regularizable in the first case, or does $f$ necessarily admit an algebraically stable model in the second?
\end{question}

\noindent{\bf Structure of the paper.} In the Section \ref{section: preliminaries}, we recall the definitions and basic properties of the first dynamical degree and the birational automorphism which are (pseudo)-regularizable or algebraically stable. In Sections~\mbox{\ref{section: example 2}, \ref{section: example 3}, \ref{section: example 4}} and \ref{section: example 5}, we prove theorems explaining Examples \ref{example 2}, \ref{example 3}, \ref{example 4} and \ref{example 5}, respectively.

\noindent{\bf AI disclosure.} The ideas of Example \ref{example 4} and the proof of Theorem \ref{theorem: non-ps-reg and non AS} were suggested by ChatGPT 6 Astra. An LLM was also used for proofreading of the text. The author independently checked every argument, wrote the final text, and takes full responsibility for its content.

\section{Preliminaries}\label{section: preliminaries}
\subsection{Definitions} Let $f\colon X\dashrightarrow Y$ be a birational map between smooth quasi-projective varieties $X$ and $Y$. Consider the greatest Zariski open subset $U\subset X$ such that $f|_{U}$ is a regular isomorphism to its image. Then { \bf the exceptional locus $\Exc(f)$} equals the union of all divisors of $X$ lying in the complement $X\setminus U$. Consider the greatest Zariski open subset $V\subset X$ such that $f|_{V}$ is a regular morphism. Then {\bf the indeterminacy locus $\Ind(f)$}  is the complement $X\setminus V$.

Note that if $Y$ is not projective or if $X$ is not normal then $\Ind(f)$ can contain divisors. However, if $f\in \Bir(X)$ is a birational automorphism of a smooth projective variety $X$ or if $f$ is a pseudo-automorphism of smooth quasi-projective $X$ then $\Ind(f)$ has codimension at least 2 in $X$.

Consider a birational automorphism $f\in \Bir(X)$ of a smooth projective variety $X$ and fix a norm $\|-\|$ on the vector space $\NS_{\RR}(X) = \NS(X)\otimes \RR$. Then the inverse image operator $f^*\colon \NS(X)\to \NS(X)$ is a bounded operator and one can compute its norm. We define the {\bf first dynamical degree of $f$} in the following way:
\[
 \lambda_1(f) = \lim_{n\to \infty} \sqrt[n]{\|(f^n)^*\|}
\]
By \cite[Theorem 1.1]{TTT_dynamical_degrees} the number $\lambda_1(f)$ does not depend on the choice of the norm $\|-\|$ on $\NS(X)$. Moreover, the dynamical degree is a birational invariant of $f\in \Bir(X)$. Thus, one can define the first dynamical degree of an automorphism of a quasi-projective variety $U$ as the dynamical degree of a birational automorphism induced on any smooth projective compactification of $U$.

\subsection{Criteria} The first dynamical degree of an algebraically stable map is very easy to compute since by the definition of the algebraic stability one has $\|(f^n)^*\| = \|(f^*)^n\| = \|f^*\|^n$; hence $\lambda_1(f)$ is the spectral radius of $f^*$. This implies the following lemma.
\begin{lemma}\label{lemma: AS implies lambda is algebraic}
 If $f\in \Bir(X)$ admits an algebraically stable model then $\lambda_1(f)$ is an algebraic integer.
\end{lemma}
A geometric criterion for algebraically stable birational automorphisms is provided in the next assertion.
\begin{lemma}\label{lemma: criterion for AS}
 Let $f\in \Bir(X)$ be a birational automorphism which is algebraically stable on a smooth projective $X$. Then the proper image of a divisor $E$ under $f^n$ is not contained in $\Ind(f)$ for all $n\geqslant 1$.
\end{lemma}
\begin{proof}
 By \cite[Lemma 2.3]{Kuznetsova_Regularizations} for any divisor class $D\in \NS(X)$ the difference $f^*((f^{n-1})^* D - (f^{n})^*D$ can be represented by the class of a divisor $E$ supported in $f^{1-n}(\Ind(f))$. Thus, by induction the proof is complete.
\end{proof}
Lemma \ref{lemma: criterion for AS} immediately implies the following assertion.
\begin{corollary}
 For any variety $X$ one has $\prPsReg(X)\subset \AS(X)$.
\end{corollary}
We will need the criterion for pseudo-regularizable birational automorphisms. The most efficient way to prove that an automorphism has no pseudo-regular model was provided in \cite{Lonjou_Urech_CAT0}. We are going to use the following corollary from the result in this paper:
\begin{theorem}[{\cite[Proposition 4.16]{Lonjou_Urech_CAT0}}]\label{theorem: CAT0}
 Let $f\in \Bir(X)$ be a birational automorphism of a smooth quasi-projective variety $X$.
 \begin{enumerate}
  \item[\textup{(1)}] The number of components of the exceptional loci $\Exc(f^n)$ is bounded by a constant which does not depend on $n$ if and only if $f$ is pseudo-regularizable. In particular, if there exists an irreducible divisor $E\subset X$ such that $f^n(E)$ is a divisor for all $n\leqslant -1$ and $\mathrm{codim}_X(f^n(E))\geqslant 2$ for all $n\geqslant 1$ then $f$ is not pseudo-regularizable.
  \item[\textup{(2)}] Assume that $X$ is a threefold and that $f$ is a pseudo-automorphism of $X$. Then the number of components of the indeterminacy loci $\Ind(f^n)$ is bounded by a constant which does not depend on $n$ if and only if $f$ is regularizable.
 \end{enumerate}
\end{theorem}

\subsection{Birational automorphisms of products} When constructing a birational automorphism that satisfies one property but not another, we sometimes give a concrete construction in a specific dimension. However, we needs examples in all dimensions. Thus, we study direct products of birational automorphisms and deduce that the properties we are interested in persist under products. For instance, the following assertion holds.
\begin{lemma}\label{lemma: if f lies in S}
 Let $f\in \Bir(X)$ and $g\in \Bir(Y)$ be birational automorphisms such that $f\in S(X)$ and $g\in S(Y)$ where $S = \prReg,\prPsReg,\Reg,\PsReg$, or $\AS$. Then  the birational automorphism $f\times g$ lies in $S(X\times Y)$.
\end{lemma}
\begin{proof}
 If  $S = \prReg,\prPsReg,\Reg$, or $\PsReg$ then the result is obvious. Now assume that $f$ and $g$ are algebraically stable birational automorphisms of smooth projective varieties $X$ and $Y$. One can check that~\mbox{$\Exc(f\times g) = \Exc(f)\times Y\cup X\times \Exc(g)$} as well as $\Ind(f\times g) = \Ind(f)\times Y\cup X\times \Ind(g)$. Then by Lemma \ref{lemma: criterion for AS} we complete the result.
\end{proof}
Once $f$ does not satisfy one of our properties it is more difficult to prove that so does its product with $g$. However, the following assertion gives us the necessary result.
\begin{lemma}\label{lemma: if f does not lie in S}
 Let $X$ and $Y$ be algebraic varieties and let $f\in \Bir(X)$ be a birational automorphism which does not lie in  $S(X)$ where $S = \prReg,\Reg$, or $\PsReg$. Then  the product automorphism $f\times \mathrm{id}_Y$ does not lie in $S(X\times Y)$.
\end{lemma}
\begin{proof}
 First, assume that $S = \PsReg$. By Theorem \ref{theorem: CAT0} the birational automorphism $f$ does not lie in the set $\PsReg(X)$ if the number of components of $\Exc(f^n)$ is not bounded. Note that $f\times  \mathrm{id}_Y$ contracts the divisor $E\times Y$ if $E\subset\Exc(f)$. Therefore, the number of components of $\Exc((f\times  \mathrm{id}_Y)^n)$ equals the number of components of $\Exc(f^n)$. Thus, $f\times \mathrm{id}_Y$ does not lie in $\PsReg(X\times Y)$.

 Note that the same reason works for $S = \Reg$ by \cite[Proposition 4.16]{Lonjou_Urech_CAT0}. Finally, if $S = \prReg$ then the result follows from \cite[Theorem 1.5]{Bedford_Kim_Pseudo}.
\end{proof}

\section{An algebraically stable automorphism which is not pseudo-regularizable}\label{section: example 2}
\begin{theorem}\label{theorem: CDX is non-regularizable}
 Let $f\in \Bir(\p^k)$ be a birational automorphism of the projective space $\p^k$ of degree $d\geqslant 2$. Denote by $g_A\in \Bir(\p^k)$ the composition $g_A  = A\circ f$ where $A\in \PGL(k+1,\KK)$. Then if $A$ is very general then the birational automorphism~$g_A$ is algebraically stable on $\p^k$; however, it is not pseudo-regularizable.
\end{theorem}

\begin{proof}
 The algebraic stability of $g_A$ for a very general $A\in \PGL(k+1,\KK)$ is proved in \cite[Theorem 1.5]{CDX}. It remains to show that $g_A$ is not pseudo-regularizable i.e. it suffices to show that the number of components of the exceptional locus $\Exc(g_A^n)$ is not bounded.

 We first recall the argument used to establish algebraic stability. Since the degree of $2$ the exceptional locus of $f$ is not empty.
 Fix an irreducible component $E$ of $\Exc(f)$ and consider its proper image $Z\subset \p^k$. Denote by $F_{E,n}$ and $G_{E,n}$ the following subsets of the group~$\PGL(k+1,\KK)$:
 \begin{align*}
  &F_{E,n} = \left\{A\in \PGL(k+1, \KK)\mid g_A^i(A(Z)) \text{ does not lie in }\Ind(f) \text{ for } 0\leqslant i\leqslant n\right\};\\
  &G_{E,n} = \left\{ A\in \PGL(k+1,\KK) \mid (A^{-1}\circ f^{-1})^{i}(A^{-1}(E)) \text{ does not lie in }\Exc(f^{-1}) \text{ for } 0\leqslant i\leqslant n\right\}.
 \end{align*}
 The sets $F_{E,n}$ and $G_{E,n}$ are Zariski open subsets of~$\PGL(k+1,\KK)$ and they are not empty for any $n\geqslant 1$ by Lemma \ref{lemma: from CDX}. In particular, the intersection $\bigcap_{n\geqslant 0} F_{E,n}\cap G_{E,n}$ is not empty in $\PGL(k+1,\KK)$ since $\KK$ is an uncountable field. Thus, a very general $A\in\PGL(k+1, \KK)$ lies in $F_{E,n}$ and in $G_{E,n}$ for any $n\geqslant 0$. Now choose $A$ in the following set
 \[
  A\in \bigcap_{E\subset \Exc(f), n\geqslant 0} F_{E,n} \cap  \bigcap_{E\subset \Exc(f), n\geqslant 0} G_{E,n}.
 \]
 Then $A$ lies in the intersection of countably many non-empty Zariski open subsets of $\PGL(k+1,\KK)$. This intersection is not empty; thus, there exists $A$ such that $g_A^n(E)$ never lies in $\Ind(g_A)$ and also the proper preimage $g_A^{-n}(E)$ is a divisor for any $E\subset \Exc(f)$ and any $ n\geqslant 1$. Therefore, $g_A$ is algebraically stable by Lemma \ref{lemma: criterion for AS} and it is not pseudo-regularizable by the first assertion of Theorem \ref{theorem: CAT0}.
\end{proof}
\begin{lemma}[{\cite[Lemma 1.6]{CDX}}]\label{lemma: from CDX}
 Let $W$ and $Q$ be non-empty, irreducible, Zariski closed subsets of $\p^k$ and let $f\in \Bir(\p^k)$ be a birational automorphism such that $\Ind(f)\subset W$. For any $n\geqslant 0$ there exists an element $A\in\PGL(k+1,\KK)$ such that $(A\circ f)^i(A(Q))$ is not contained in $W$ for all $i = 0,\dots, n$.
\end{lemma}

\section{A non-regularizable pseudo-automorphism of a projective threefold}\label{section: pr ps-reg non-reg} \label{section: example 3}
Let $Q\subset \p^3$ be a smooth cubic surface. Then to each point $p\in Q$ one can associate a birational involution~$\sigma_p$ of $\p^3$ in the following way. We have to define the image of a general point $q\in \p^3$. In order to do it we consider a line $L$ in the projective space $\p^3$ passing through points $p$ and $q$. Since $q$ is general the line $L$ intersects $Q$ by exactly three distinct points $L\cap Q = \{p,p_1,p_2\}$. There exists a unique non-trivial involution $\sigma_L$ of the projective line $L\cong \p^1$ such that $\sigma_L(p_i) = p_i$ for $i = 1$ and $2$. We define $\sigma_p(q) = \sigma_L(q)$. One can easily deduce that $\sigma_p$ is an involution.

This class of birational involutions of $\p^3$ was studied in \cite{Blanc_Pseudo}. In particular, it was proved there that the restricted automorphism $\sigma_p|_{Q}$ is the identity i.e. $\sigma_p$ lies in the inertia group of $Q$. Moreover, the indeterminacy locus $\Ind(\sigma_p)$ has only one component; it is a curve $\Gamma_p$ lying on the surface $Q$ and having a unique singular point $\Sing(\Gamma_p) = p$. Let $X_p$ be the variety obtained by blowing up $p$ in $\p^3$ and then blowing up the strict transform of $\Gamma_p$ in $\mathrm{Bl}_p(\p^3)$. Then $\sigma_p$ induces a pseudo-automorphism of a smooth threefold $X_p$.

Moreover, if we consider several points $p_1,\dots, p_k\in Q\subset \p^3$ and denote by $X$ the successive blow-up of $\p^3$ in points $p_1,\dots, p_k$ and then in strict transforms of the curves $\Gamma_{p_1},\dots, \Gamma_{p_k}$ then $f = \sigma_{p_k}\circ\dots\circ \sigma_{p_1}$ induces a pseudo-automorphism of $X$, see \cite[Proposition 2.2]{Blanc_Pseudo}.

In the case $Q$ is very general, $k = 3$ and points $p_1,p_2$ and $p_3$ are sufficiently general the pseudo-automorphism $f = \sigma_{p_3}\circ\sigma_{p_2}\circ\sigma_{p_1}$ is not projectively regularizable by \cite[Theorem 1.3]{Kuznetsova_Regularizations}. Our goal is to extend this result slightly.

\begin{theorem}\label{theorem: pr ps-reg non-reg}
 Let $Q\subset \p^3$ be a very general smooth cubic surface and let $p_1, p_2, p_3$ be sufficiently general points on $Q$. Then the composition $f = \sigma_{p_3}\circ \sigma_{p_2}\circ\sigma_{p_1}$ induces a pseudo-automorphism of $X$ and it is not regularizable.
\end{theorem}
\begin{proof}
Let $Q$ and $p_1,p_2,p_3$ be the smooth cubic surface and the points described in Theorem \ref{theorem: pr ps-reg non-reg}. Consider the line $L$ in $\p^3$ passing through the point $p_1$ and a point in the intersection $\Gamma_{p_1}\cap \Gamma_{p_2}$ of curves on the surface~$Q$. By \cite[Corollary 4.30]{Kuznetsova_Regularizations} the curve $L$ does not lie in the indeterminacy loci $\Ind(f^{-n})$ for all $n\geqslant 1$. Denote by $C_n$ the proper image of $L$ under $f^{-n}$. Then $C_n$ does not lie in $\Ind(f^n)$ for all $n\geqslant 1$.

By \cite[Corollary 4.9]{Kuznetsova_Regularizations} and \cite[Lemma 2.11]{Kuznetsova_Regularizations} there exists an integer $N\geqslant 1$ such that $L\subset \Ind(f^m)$ for all $m\geqslant N$. Thus, the curves $C_n$ for $n\geqslant 1$ are all distinct since in the case $C_i = C_j$ where $i<j$ there is an $f$-periodic orbit of curves $C_i,\dots, C_{j-1}$ and none of them lies in $\Ind(f^k)$ for all $k\geqslant 1$.

Therefore, if $k\geqslant N$ then the indeterminacy locus $\Ind(f^k)$ contains the curves $L, C_1, \dots, C_{k-N+1}$. Thus, the locus $\Ind(f^k)$ has at least $k-N$ irreducible components. By the second assertion of Theorem \ref{theorem: CAT0} we conclude that $f$ is not regularizable.
\end{proof}

\section{A pseudo-automorphism with no algebraically stable model} \label{section: example 4}
\begin{theorem}\label{theorem: ps-reg and non-AS}
 Let $A \in \SL(3,\ZZ)$ be an integer matrix such that its characteristic polynomial is irreducible over $\QQ$ and its roots are complex numbers $\alpha, \overline{\alpha}$ and $\beta$ such that $\alpha$ is non-real and $|\beta|<1<|\alpha|$. Then $A$ defines an automorphism $f$ of the algebraic torus $\mathbb{G}_m^3$ with no algebraically stable model.
\end{theorem}
\begin{lemma}[{\cite[Section 7.2]{BDJK}}]\label{lemma: example of matrix A}
 Let $A$ be the following matrix:
 \[
  A = \begin{pmatrix}
       0&-1&1\\ 1&0&0\\ 0&1&0
      \end{pmatrix}
 \]
 Then $A$ is an element of $\SL(3,\ZZ)$ which satisfies the condition in Theorem \textup{\ref{theorem: ps-reg and non-AS}}.
\end{lemma}

\begin{lemma}\label{lemma: nice property of alpha}
 Let $A$ be a matrix as in Theorem \textup{\ref{theorem: ps-reg and non-AS}}. Then $\alpha/|\alpha|$ is not a root of unity.
\end{lemma}
\begin{proof}
 Since $\det(A) = 1$ one has $0< \beta = |\alpha|^{-2}<1$. If  $\alpha/|\alpha|$ is a root of unity then $\alpha^m/|\alpha|^m = 1$. Thus, the number $\alpha^m$ is real. Consider the eigenvalues of the matrix $A^m$, they are $\alpha^m,\overline{\alpha}^m = \alpha^m$ and $\beta^m$.  Thus, the characteristic polynomial of $A^m$ has a double root and is of degree $3$. Therefore, this polynomial is not irreducible and $\alpha^m$ and $\beta^m$ are integers. This is impossible since $0<\beta<1$.
\end{proof}

\begin{lemma}\label{lemma: no periodic divisors in T}
 Let $A\in \SL(3,\ZZ)$ be an integer matrix such that its characteristic polynomial is irreducible over $\QQ$. Then $A$ defines an automorphism $f$ of the algebraic torus $T = \mathbb{G}_m^3$ such that each non-trivial divisor $D\subset T$ is not periodic.
\end{lemma}
\begin{proof}
 Denote by $x_1,x_2,x_3$ the coordinates of the torus $T$ in which $f$ acts as the matrix $A$. Any irreducible divisor $D\subset T$ equals the zero locus of an irreducible polynomial $ P(x) = \sum_I m_I x^I$. The inverse image of $D$ under $f$ is the zero locus of the following polynomial:
 \[
  f^*(P)(x) = \sum_I m_I x^{A\cdot I}.
 \]
 If $f^N(D)$ coincides with $D$ in $T$ for some $N\geqslant 1$ then  $f^*(P) = x^K P$ where $K = (k_1,k_2,k_3)^t$ is an integer vector. This implies that the following two sets are equal:
 \[
  S = \{I \mid m_I\ne 0\}  = \{A^N\cdot I - K \mid m_I\ne 0\}.
 \]
 Then the rule $\sigma(I) = A^N\cdot I - K$ defines an endomorphism of the set $S$. Note that if one has $\sigma(I') = \sigma(I'')$ then $I'  - I''$ lies in the kernel of the matrix $A^N$. Since $A$ is invertible this implies that $\sigma$ defines a permutation of degree $M$ of the finite set $S$.

 Set $K' = \sum_{j=0}^{M-1} A^j\cdot K$. Then we deduce that $I = A^{NM}\cdot I - K'$ for all $I$ such that $m_I\ne 0$. Thus, if there are at least two  monomials $x^{I'}$ and $x^{I''}$ in $P$ with non-zero coefficients then $I' - I''$ is a non-zero eigenvector of $A^{NM}$ with eigenvalue $1$. This is possible only if the characteristic polynomial of $A$ is irreducible and cyclotomic of degree $3$ which never happens.

 Thus, we get a contradiction and there is at most one non-zero monomial in $P$. Then $P$ defines a trivial divisor on $T$. This completes the proof.
\end{proof}

\begin{corollary}\label{corollary: if f is conjugated by a to AS then a is ps-iso}
 Let $A\in \SL(3,\ZZ)$ be an integer matrix such that its characteristic polynomial is irreducible over $\QQ$ and let $f$ be an automorphism of $T = \mathbb{G}_m^3$ defined by $A$. If  $\alpha\colon T\dashrightarrow X$ is a birational map to a projective birational model $X$ of $T$ such that $g = \alpha\circ f\circ\alpha^{-1}$ is an algebraically stable birational automorphism~$X$, then $\alpha$ does not contract divisors in $T$.
\end{corollary}
\begin{proof}
 Assume that $D\subset T$ is a divisor whose proper image $Z\subset X$ under $\alpha$ is of codimension at least 2. Denote by $D_n = f^n(D)$ the image of $D$ under $f^n$ for $n\in \ZZ$ and by $Z_n$ the proper image of $D_n$ under $\alpha$. Since $\Exc(\alpha)$ can have only finitely many components, there exists strictly positive numbers $M_1, M_2$ such that for all $n\geqslant M_1$ or $n\leqslant -M_2$ the subvariety $Z_n\subset X$ is a divisor. Then $g^{M_2}$ contracts the divisor $Z_{-M_2}$ to the indeterminacy locus of $g^{M_1}$. By Lemma \ref{lemma: criterion for AS} we get a contradiction with the assumption that $g$ is algebraically stable.
\end{proof}

\begin{lemma}\label{lemma: exact sequence for NS}
 Let $T\cong \mathbb{G}_m^3$ be an algebraic torus and $\alpha\colon T\dashrightarrow X$ be a birational map which contracts no divisors in $T$ to a smooth projective birational model $X$ of $T$. If $\Exc(\alpha^{-1}) = \bigcup_{i=1}^k D_i$ where $D_i$ is an irreducible divisor in $X$ for $i = 1,\dots,k$ then the Neron--Severi group $\NS(X)$ fits the following exact sequence:
 \begin{equation}\label{equation: exact sequence for NS}
  0\to M \to \langle D_1,\dots, D_k\rangle \xrightarrow{r}  \NS(X)\to 0,
 \end{equation}
 where $r$ maps $D_i$ to its class in the group $\NS(X)$ and $M\cong \ZZ^3$ is the lattice of characters of the torus $T$.
\end{lemma}
\begin{proof}
 Denote by $U$ the complement $X\setminus \bigcup_{i=1}^k D_i$. Then $\alpha\colon T\dashrightarrow U$ is an isomorphism in codimension $1$. Then $\alpha$ induces an isomorphism of the rings of regular functions $\KK[U]\cong \KK[T]\cong \KK[x_1^{\pm 1},x_2^{\pm 1}, x_3^{\pm 1}]$ and any divisor on $U$ is principal. Thus, any integer divisor $D$ on $X$ is algebraically equivalent to a divisor $\sum_{i=1}^k \mu_i D_i$ where $\mu_i\in \ZZ$. Therefore, the map $r$ in the exact sequence \eqref{equation: exact sequence for NS} is surjective.

 The linear combination $D = \sum_{i=1}^k \mu_i D_i$ lies in the kernel of $r$ if it is defined a principal divisor. Thus, there exists a function $h\in \KK(X)$ such  that $D = (h)$. Thus, the rational function $h$ has no zeros and poles in~$U$. Thus, $h$ is a monomial $h = x_1^{m_1}x_2^{m_2} x_3^{m_3}$ where $(m_1,m_2,m_3)\in \ZZ^3$ and $\Ker(r) = M$.
\end{proof}

\begin{lemma}\label{lemma: action on space of D_i}
  In the setting of Corollary \textup{\ref{corollary: if f is conjugated by a to AS then a is ps-iso}} if $D_1,\dots,D_k$ are irreducible divisor components in $\Exc(\alpha^{-1})$ then the inverse image of the birational automorphism $g = \alpha\circ f\circ\alpha^{-1}$ induces a linear action $g^{\#}$ on the real vector space $\langle D_1,\dots, D_k\rangle$ which preserves the cone $\RR_{\geqslant 0 }D_1 + \dots + \RR_{\geqslant 0} D_k$.
\end{lemma}
\begin{proof}
 In order to define the action $g^{\#}$ on the space $\langle D_1,\dots, D_k\rangle$ we fix an elimination of indeterminacies of the birational automorphism $g$:
\[\begin{tikzcd}
	& V & \\
	X && X
	\arrow["p"', from=1-2, to=2-1]
	\arrow["q", from=1-2, to=2-3]
	\arrow["g", dashed, from=2-1, to=2-3]
\end{tikzcd}\]
Here $V$ is a smooth projective variety, $p$ and $q$ are birational morphisms and the diagram commutes. Consider the divisor $g^*([D_i]) = p_*q^*[D_i]$. Note that $q^*[D_i] = [\widetilde{D_i}] + \sum_{r=1}^m b_r [E_r]$ where $\widetilde{D_i}$ is the strict transform of $D_i$ and $E_1,\dots,E_m$ are components of the exceptional locus $\Exc(q)$. Note that the numbers $b_r$ are non-negative by Mori negativity lemma. To complete the proof it remains to show that the subvarieties $p(\widetilde{D_i})$ and $p(E_r)$ for all~\mbox{$1\leqslant r\leqslant m$} are of codimension at least $2$ in $X$ or lie in $\bigcup_{i=1}^k D_i$.

Denote by $U$ the complement $X\setminus \bigcup_{i=1}^k D_i$. By Corollary \ref{corollary: if f is conjugated by a to AS then a is ps-iso} the birational map $\alpha$ induces an isomorphism in codimension $1$. Since $f$ is a regular automorphism of the torus $T$ we deduce that $g|_{U}$ is a pseudo-automorphism. Therefore, for all $1\leqslant i\leqslant k$ the proper inverse image of $D_i$ is either of codimension at least~$2$ or coincides with $D_j$ for $1\leqslant j \leqslant k$. One can also deduce that $\Exc(g)\subset \bigcup_{i=1}^k D_i$.

 Define $g^{\#}D_i = p(\widetilde{D_i})  + \sum_{r=1}^m b_r p(E_r)$ which is a divisor in $\ZZ_{\geqslant 0 }D_1 + \dots + \ZZ_{\geqslant 0} D_k$. This rule defines a linear operator  $g^{\#}$ of  the space $\langle D_1,\dots, D_k\rangle$ which preserves the cone $\RR_{\geqslant 0 }D_1 + \dots + \RR_{\geqslant 0} D_k$. By construction one has $g^*[D_i] = [g^{\#} D_i]$.
\end{proof}

\begin{proof}[Proof of Theorem \textup{\ref{theorem: ps-reg and non-AS}}] Assume that $X$ is a smooth projective threefold and $\alpha\colon T\dashrightarrow X$ is a  birational map such that $g = \alpha\circ f\circ\alpha^{-1}$ is an algebraically stable map. Denote by $D_1,\dots, D_k$ the exceptional divisors for~$\alpha^{-1}$. By Lemma \ref{lemma: exact sequence for NS} the Neron--Severi group of $X$ fits the exact sequence \eqref{equation: exact sequence for NS}.

By Lemma \ref{lemma: action on space of D_i} there exists a linear action $g^{\#}$ on the space $\langle D_1,\dots, D_k\rangle$ which commutes with the standard inverse image action of $g^*$ on the Neron--Severi group i.e. $g^*(r(D))=r(g^{\#}(D))$. On the kernel $M$ of the map $r$ the operator $g^{\#}$ induces the standard action of the matrix $A$ on the character lattice.

Since $g$ is algebraically stable the operator $g^*|_{\NS(X)}$ has an eigenvalue $\lambda_1(g) = |\alpha|$ and this eigenvalue has the maximal absolute value on  $\NS(X)$. Since the matrix $A$ has an eigenvalue $\alpha$ then $g^{\#}|_{\langle D_1,\dots, D_k\rangle}$ has eigenvalues $\alpha$ and $|\alpha|$ and has no eigenvalues of greater absolute value. By Lemma \ref{lemma: action on space of D_i} the operator~$g^{\#}$ is defined by a non-negative matrix in the standard basis.  Thus, by Perron--Frobenius theorem the number~$\alpha/|\alpha|$ is a root of unity; see, \cite[Theorem 8.3.4(iv)]{Perron_Frobenious_book}. However, this contradicts Lemma \ref{lemma: nice property of alpha} and the proof is complete.
\end{proof}
Theorem \ref{theorem: ps-reg and non-AS} gives a three-dimensional example of a regular automorphism which admits no algebraically stable model. The following assertion extends this example to all higher dimensions.
\begin{corollary}\label{corollary:  ps-reg and non-AS in any dimension}
 Let $h\in \Aut(\mathbb{G}_m^{r})$ be the automorphism of an algebraic torus with coordinates $y_1,\dots, y_r$ such that $h(y_1,\dots, y_r) = (\mu_1 y_1,\dots, \mu_r y_r)$, where $\mu_1,\dots, \mu_r$ are non-zero complex numbers such that $\mu_1^{i_1}\dots \mu_r^{i_r} = 1$ for $i_1,\dots,i_r\in \ZZ$ only if $i_1 = \dots = i_r = 0$. Let $f\in \Aut(\mathbb{G}_m^3)$ be as described in Theorem \textup{\ref{theorem: ps-reg and non-AS}}. Then the automorphism $f\times h\in \Aut(\mathbb{G}_m^{r+3})$ admits no algebraically stable model.
\end{corollary}
\begin{proof}
 By Lemma \ref{lemma: f times h-preiodic orbits} the torus $\mathbb{G}_m^{r+3}$ contains no $f\times h$-periodic divisors. Then the proof is analogous to the proof of Theorem \ref{theorem: ps-reg and non-AS}.
\end{proof}
The following lemma is an easy corollary from Lemma \ref{lemma: no periodic divisors in T}.
\begin{lemma}\label{lemma: f times h-preiodic orbits}
 Let $h\in \Aut( \mathbb{G}_m^r)$ and $f\in \Aut(\mathbb{G}_m^3)$ be as described in Corollary \textup{\ref{corollary: ps-reg and non-AS in any dimension}}. Then the orbit of any divisor in $\mathbb{G}_m^{r+3}$ under the action of the automorphism $f\times h\in \Aut(\mathbb{G}_m^{r+3})$ is infinite.
\end{lemma}
\begin{proof}
 A divisor $D$ in $\mathbb{G}_m^{r+3}$ is the zero locus of an irreducible polynomial $P(x,y) = \sum_{I,J} m_{I,J}x^Iy^J$. The image of $D$ under $f\times h$ is defined by the following polynomial:
 \[
  (f\times h)^*(P)(x,y) = \sum_{I,J} \mu^J m_{I,J}x^{AI}y^J.
 \]
 Analogously to the proof of Lemma \ref{lemma: no periodic divisors in T} we deduce that if $D$ is periodic then there exists at most one index $I$ for which some $J$ satisfies $m_{I,J}\ne 0$. Therefore, $X^I$ is the divisor of $P$ and $I = (0,0,0)$ since $P$ is irreducible. We conclude that $P(X,Y) = P(Y)$; since the numbers $\mu_1,\dots, \mu_r$ are multiplicatively independent this implies that $P(Y)$ is a monomial in $y_1,\dots, y_r$. This completes the proof.
\end{proof}

\section{A non-pseudo-regularizable automorphism with no algebraically stable model}\label{section: example 5}
We recall the construction of a birational automorphism of $\p^3$ whose dynamical degree is a transcendental number introduced in \cite{BDJK}.
Given an element $A\in \SL(3,\ZZ)$ one can define the automorphism of the algebraic torus $h_A(z) = z^A$ where $z = (z_1,z_2,z_3)$ are the coordinate on the torus. The automorphism $h_A$ extends to an element $h_A\in \Bir(\p^3)$ where we consider on $\p^3$ the homogeneous coordinates $(x_0:x_1:x_2:x_3)$ such that $z_i = x_i/x_0$.

Let $J\in \Bir(\p^3)$ be the Cremona involution $J(x_0:x_1:x_2:x_3) = \left(\frac{1}{x_0}:\frac{1}{x_1}:\frac{1}{x_2}:\frac{1}{x_3}\right)$ and let $B$ be the element in $\PGL(4,\KK) = \Aut(\p^3)$ defined by the following matrix:
\begin{equation}\label{equation: matrix B}
 B = \begin{pmatrix}
      1&-1&1&-1\\ 1&1&-1&1\\-1&1&1&-1\\1&-1&1&1
     \end{pmatrix}.
\end{equation}
We denote by $g$ the birational involution $g = B^{-1}\circ J\circ B\in \Bir(\p^3)$. Then by \cite[Theorem 1.1]{BDJK} there exists $A\in \SL(3,\ZZ)$ such that the first dynamical degree of the birational automorphism $f = g\circ h_A$ is a transcendental number. In view of \cite[Section 7.2]{BDJK} one can take the following matrix $A$:
\begin{equation}\label{equation: matrix A for transcendental}
 A = \begin{pmatrix}
      -3&-14 &-12\\ 4&11&6\\-2&-4 &-1
     \end{pmatrix}.
\end{equation}
By Lemma \ref{lemma: AS implies lambda is algebraic} the birational automorphism $f$ does not admit an algebraically stable model. Our goal is to show that it is not pseudo-regularizable either.
\begin{theorem}\label{theorem: non-ps-reg and non AS}
 Let $A$ be the matrix defined in \textup{\cite[Section 7.2]{BDJK}}. Then the composition $f = g\circ h_A$ is a non-pseudo-regularizable birational automorphism of $\p^3$ which does not admit an algebraically stable model.
\end{theorem}
In order to prove the theorem we first need to show several technical results. First, we need the following assertion which describes the orbit of the basis vector $e_1$ under the action of the matrix $A$.
\begin{lemma}[{\cite[Lemma 7.4]{BDJK}}]
 Let $A$ be as in \textup{\eqref{equation: matrix A for transcendental}} and let $e_1$ be the vector $e_1 = (1,0,0)^{t}$. If we denote $(A^n e_1)^t = (a_{n1},a_{n2}, a_{n3}) $ then the numbers $0,a_{n1}, a_{n2},a_{n3}$ are pairwise distinct for all $n\geqslant 1$.
\end{lemma}
\begin{proof}
 This assertion is proved in \cite{BDJK} but it is not formulated as a separate lemma there. Here we show how to deduce this fact from results in that paper.

 In \cite[Section 7.2]{BDJK} the authors define the matrix $A$. In order to show that the matrix $A$ works well for their purposes they need to prove that $A$ satisfies the condition of \cite[Theorem 3.6]{BDJK}.

 Denote by $\Sigma(\p^3)\subset \RR^3$ the standard toric fan of $\p^3$.  The condition of \cite[Theorem 3.6]{BDJK} is that $A^n v$ lies in the interior of $3$-dimensional cones of $\Sigma(\p^3)$ for vectors $v$ in the finite set of vectors including $e_1$ and all $n\geqslant 1$. In order to show that the matrix $A$ satisfies this condition the authors check a stronger result in \cite[Lemma 7.4]{BDJK}; namely, they prove that $A^n v$ does not lie in any plane supporting $2$-dimensional cones of $\Sigma(\p^3)$ for all $v$ in their finite set and all $n\geqslant 1$.

 Therefore, we conclude that for all $n\geqslant 1$ the vector $A^ne_1$ does not lie in planes $\{x_i = 0\}$ for $i = 1,2,3$ and in planes $\{x_i = x_j\}$ for $1\leqslant i<j\leqslant 3$. Then the result follows immediately.
\end{proof}
\begin{lemma}\label{lemma: preimages of v}
 Let $A$ be as in \textup{\eqref{equation: matrix A for transcendental}} and let $v$ be the vector $v = (1,1,0)^{t}$. If we denote $(A^{-n} v)^t = (a_{n1},a_{n2}, a_{n3}) $ then $a_{n1}, a_{n3}>0$ and $a_{n2}<0$ for all $n\geqslant 1$.
\end{lemma}
\begin{proof}
 Denote by $R$ and $C$ the following matrices and compute the product $RC$:
 \begin{align*}
  R = \begin{pmatrix}
       1 &-2&3\\0&1&-2\\0&0&1
      \end{pmatrix};&&
  C = \begin{pmatrix}
       3&2&4\\4&3&6\\6&4&9
      \end{pmatrix}; &&
  RC = \begin{pmatrix}
        13&8&19\\-8&-5&-12\\6&4&9
       \end{pmatrix}.
 \end{align*}
 One can check that $A^{-1} = RCR^{-1}$ and $R^{-1}v = (3,1,0)^t$. Then $A^{-n}v  = (RC)C^{n-1}(3,1,0)^t$ for all $n\geqslant 1$. Since all entries of $C$ are strictly positive so are entries of $C^{n-1}$. Thus, the coordinates of the vector $C^{n-1}v$ are strictly positive. Since all entries in the first and the third row of $RC$ are strictly positive and the entries in its second row are strictly negative we complete the proof.
\end{proof}

The automorphism $f = g\circ h_A$ is defined over $\ZZ$; thus, one can consider its action over the field $\QQ(t)$ where $t$ is transcendental over $\QQ$. Consider the standard valuation $\nu\colon \QQ(t)^{\times}\to \ZZ$ such that $\nu(c) = 0$ for all $c\in \QQ^{\times}$ and $\nu(t^m) = m$ for all $m\in \ZZ$. Define the map $\tau\colon \p^3_{\QQ(t)}\to \ZZ^3$ in the following way; for a point~$(x_0:x_1:x_2:x_3)\in \p^3_{\QQ(t)}$ we set:
\begin{equation}\label{equation: tau}
 \tau(x_0:x_1:x_2:x_3) = (\nu(x_1)-\nu(x_0),\nu(x_2)-\nu(x_0),\nu(x_3)-\nu(x_0)).
\end{equation}
This map is well-defined on the stated open subset of projective space: rescaling all homogeneous coordinates by a common nonzero scalar leaves the valuation differences unchanged. Moreover, $\tau$ is compatible with birational automorphisms $h_A$ and $g$ defined above.
\begin{lemma}\label{lemma: valuation of images of a point}
 If $x$ is a point $x = (x_0:x_1:x_2:x_3)\in \p^3_{\QQ(t)}\setminus \bigcup_{i=0}^3\{x_i = 0\}$ then  $\tau(h_A(x)) = A\tau(x)$. Moreover, if $\tau(x) = (\tau_1,\tau_2,\tau_3)^t$ where $\min\{0,\tau_1,\tau_2,\tau_3\}$ is attained uniquely $x$ does not lie in $\Ind(g)$ and one has $\tau(g(x)) = \tau(x)$.
\end{lemma}
\begin{proof}
  One can choose the homogeneous coordinates of $x$ so that $x = (1:z_1:z_2:z_3)$ where $z_1,z_2,z_3\in \QQ(t)$. Then $h_A(x) = (1: z^{A e_1}: z^{Ae_2}: z^{Ae_3})$ where $e_1,e_2,e_3$ form the standard basis in $3$-dimensional space. If we denote the entries of $A$ by $(a_{ij})$ one can check that $\nu(z^{A e_i}) = \nu(z_1^{a_{i1}}z_2^{a_{i2}}z_3^{a_{i3}}) = \sum_{j=1}^3 a_{ij}\nu(z_j)$. Therefore, the proof of the equality $\tau(h_A(x)) = A\tau(x)$ is complete.

 Recall that $g = B^{-1}\circ J\circ B$ where $B\in \PGL(4,\QQ)$ is defined in \eqref{equation: matrix B} and $J$ is the Cremona involution of $\p^3$. For $i = 0,\dots, 3$ denote by $b_i$ the $i$-th component of the vector $B(1,z_1,z_2,z_3)^t$. Then one can easily check the following equalities:
 \begin{align}\label{equation: bi+bj}
  &b_0+b_1 = 2; && b_1+b_2 = 2z_1;&& b_2+b_3 = 2z_2; && b_3 - b_0 = 2z_3.
 \end{align}
 By assumption $\min\{0,\tau_1,\tau_2,\tau_3\}$ is attained uniquely; thus, $\nu(b_i) = \min(0,\nu(z_1),\nu(z_2),\nu(z_3)) = \nu(b_j)$ for all~\mbox{$i,j = 0,\dots,3$}. Moreover, since coordinates $z_1,z_2$ and $z_3$ are non-zero so are $b_i$ for $i = 0,\dots, 3$. Now consider $\tau(g(x))$:
 \begin{multline*}
  \tau(g(1:z_1:z_2:z_3)) = \tau(B^{-1}(J(b_0:b_1:b_2:b_3)))) = \tau\left(B^{-1}\left(\frac{1}{b_0}:\frac{1}{b_1}:\frac{1}{b_2}:\frac{1}{b_3}\right)\right) = \\
  = \tau\left(\frac{b_0+b_1}{b_0b_1}: \frac{b_1+b_2}{b_1b_2}: \frac{b_2+b_3}{b_2b_3}: \frac{b_3-b_0}{b_0b_3}\right)
  = \tau\left(1: \frac{b_0(b_1+b_2)}{b_2(b_0+b_1)} : \frac{b_0b_1(b_2+b_3)}{b_2b_3(b_0+b_1)}: \frac{b_1(b_3-b_0)}{b_3(b_0+b_1)}\right) = \\ =
  \tau\left(1: z_1\frac{b_0}{ b_2}: z_2\frac{b_0b_1  }{b_2b_3}: z_3\frac{b_1}{b_3}\right).
 \end{multline*}
 The last equality is due to the observation \eqref{equation: bi+bj}. Since in our assumption the valuation $\nu$ has the same value on all $b_i$ we conclude that $\tau(g(x)) =\tau(x)$. This completes the proof.
\end{proof}
In view of Lemma \ref{lemma: valuation of images of a point} we can find a point in $\p^3_{\QQ(t)}$ whose forward $f$-orbit avoids sets $\Ind(f)$ and $\Exc(f)$.
\begin{corollary}\label{corollary: the point p and its images}
 Fix the point $p = (1:t:1:2-t)$ in  $\p^3_{\QQ(t)}$. For any $n\geqslant 0$ the point $p$ does not lie in $\Ind(f^n)$ and $f^n(p)$ does not lie in $\Exc(f)$.
\end{corollary}
\begin{proof}
 Note that the point $p$ lies in $\p^3_{\QQ(t)}\setminus \bigcup_{i=0}^3\{x_i = 0\}$ and $\tau(p)  = (1,0,0)^t = e_1$. Thus, by Lemma \ref{lemma: valuation of images of a point} we have $\tau(h_A(p)) = A e_1$. In particular, the coordinates of $\tau(h_A(p))$ are pairwise distinct and none of them vanishes. Therefore, by Lemma \ref{lemma: valuation of images of a point} the point $h_A(p)$ does not lie in $\Ind(g)$ and $\tau(f(p)) = \tau(g(h_A(p))) = Ae_1$. Proceeding by induction we deduce that all coordinates of $f^n(p)$ are non-zero, that $\tau(f^n(p)) = A^ne_1$; therefore, all coordinates of $\tau(f^n(p))$ are pairwise distinct and non-zero. This implies that $f^n(p)$ lies neither in $\Ind(f)$ nor in $\Exc(f)$ for all $n\geqslant 0$. This completes the proof.
\end{proof}

\begin{corollary}\label{corollary: point on H and its backword images}
 Consider the divisor $H = \{x_0 - x_1+x_2-x_3 = 0\}\subset \p^3_{\mathbb{Q}(t)}$. Then $f^{-n}(H) = (0:0:1:0)$ for all $n\geqslant 1$.
\end{corollary}
\begin{proof}
 Let $q = (x_0(t):x_1(t):x_2(t):x_3(t))$ be a very general point on the divisor $H$ over the field $\mathbb{Q}(t)$. One can compute $g(H) = (Q(t):0:0:-Q(t))$; thus, $\tau(g(H))$ is proportional to the vector $(1,1,0)^t$. By Lemmas \ref{lemma: preimages of v} and \ref{lemma: valuation of images of a point} we deduce that $\tau(f^{-n}(H))$ is proportional to the vector $v_n = (a_{n1},a_{n2}, a_{n3}) $ such that~\mbox{$a_{n1}, a_{n3}>0$} and $a_{n2}<0$ for all $n\geqslant 1$. Thus, all points in $f^{-n}(H)$ are $(x_0(t):x_1(t):x_2(t):x_3(t))$ such that $\nu(x_i)>\nu(x_0)>\nu(x_2)$ for $i =1,3$. The set $f^{-n}(H)$ is a Zariski closed subset of $\p^3$ and since $f$ is defined over $\ZZ$ then so is $f^{-n}(H)$. Therefore, the only point which satisfies this description is $(0:0:1:0)$. This completes the proof.
\end{proof}
Now we are ready to prove Theorem \ref{theorem: non-ps-reg and non AS}.
\begin{proof}[Proof of Theorem \textup{\ref{theorem: non-ps-reg and non AS}}]
 By \cite[Section 7.2]{BDJK} the dynamical degree of $f$ is a transcendental number. Thus, by Lemma \ref{lemma: AS implies lambda is algebraic} the birational automorphism $f$ does not admit an algebraically stable model.

 It remains to show that $f$ is non-pseudo-regularizable. By the first assertion of Theorem \ref{theorem: CAT0} in order to do this it suffices to find a divisor $H\subset\p^3$ such that the proper image of $H$ under $f^n$ is a divisor for $n\geqslant 1$ and is a point for $n\leqslant -1$. Consider the following divisor $H$:
 \[
  H = \{x_0 - x_1+x_2-x_3 = 0\}.
 \]
 By Corollary \ref{corollary: point on H and its backword images} we deduce that $f^{n}(H) = (0:0:1:0)$ for all $n\leqslant -1$.

 Now consider the point $p = (1:t:1:2-t)$ where $t$ is any transcendental number in $\KK$. On the one hand, the point $p$ lies in $H$. On the other hand, by Corollary \ref{corollary: the point p and its images} for any $n\geqslant 0$ the point $p$ does not lie in~$\Ind(f^n)$ and $f^n(p)$ does not lie in $\Exc(f)$. Therefore, the proper image of $H$ under $f^n$ is a hypersurface for all $n\geqslant 1$. This completes the proof.
\end{proof}

\bibliographystyle{alpha}
\bibliography{references}

@article{AX,
AUTHOR = {Abboud, M. and Xie, J.},
TITLE = {Dynamical degrees of twisted rational maps},
 JOURNAL = {\rm  arXiv:2608.09275},
      YEAR = {2026},
}

@article {Equivariant_resolution,
    AUTHOR = {Abramovich, D. and Wang, J.},
     TITLE = {Equivariant resolution of singularities in characteristic
              {$0$}},
   JOURNAL = {Math. Res. Lett.},
  FJOURNAL = {Mathematical Research Letters},
    VOLUME = {4},
      YEAR = {1997},
    NUMBER = {2-3},
     PAGES = {427--433}
}

@incollection {Bayle_Beauville,
    AUTHOR = {Bayle, L. and Beauville, A.},
     TITLE = {Birational involutions of {${\bf P}^2$}},
      NOTE = {Kodaira's issue},
   JOURNAL = {Asian J. Math.},
  FJOURNAL = {Asian Journal of Mathematics},
    VOLUME = {4},
      YEAR = {2000},
    NUMBER = {1},
     PAGES = {11--17},
}

@article {Bedford_Kim_Pseudo,
    AUTHOR = {Bedford, E. and Kim, K.},
     TITLE = {Dynamics of (pseudo) automorphisms of 3-space: periodicity
              versus positive entropy},
   JOURNAL = {Publ. Mat.},
  FJOURNAL = {Publicacions Matem\`atiques},
    VOLUME = {58},
      YEAR = {2014},
    NUMBER = {1},
     PAGES = {65--119},
}

@article {BCK,
    AUTHOR = {Bedford, E. and Cantat, S. and Kim, K.},
     TITLE = {Pseudo-automorphisms with no invariant foliation},
   JOURNAL = {J. Mod. Dyn.},
  FJOURNAL = {Journal of Modern Dynamics},
    VOLUME = {8},
      YEAR = {2014},
    NUMBER = {2},
     PAGES = {221--250},
}

@article {BDJK,
    AUTHOR = {Bell, J. and Diller, J. and Jonsson, M. and
              Krieger, H.},
     TITLE = {Birational maps with transcendental dynamical degree},
   JOURNAL = {Proc. Lond. Math. Soc. (3)},
  FJOURNAL = {Proceedings of the London Mathematical Society. Third Series},
    VOLUME = {128},
      YEAR = {2024},
    NUMBER = {1},
     PAGES = {Paper No. e12573, 47},
      ISSN = {0024-6115,1460-244X},
}

@article {Blanc_Pseudo,
    AUTHOR = {Blanc, J.},
     TITLE = {Dynamical degrees of (pseudo)-automorphisms fixing cubic
              hypersurfaces},
   JOURNAL = {Indiana Univ. Math. J.},
    VOLUME = {62},
      YEAR = {2013},
    NUMBER = {4},
     PAGES = {1143--1164},
}

@article {Blanc-Cantat,
    AUTHOR = {Blanc, J. and Cantat, S.},
     TITLE = {Dynamical degrees of birational transformations of projective
              surfaces},
   JOURNAL = {J. Amer. Math. Soc.},
  FJOURNAL = {Journal of the American Mathematical Society},
    VOLUME = {29},
      YEAR = {2016},
    NUMBER = {2},
     PAGES = {415--471},
}

@book {Perron_Frobenious_book,
    AUTHOR = {Brualdi, R. A. and Cvetkovi\'{c}, D.},
     TITLE = {A combinatorial approach to matrix theory and its
              applications},
    SERIES = {Discrete Mathematics and its Applications (Boca Raton)},
 PUBLISHER = {CRC Press, Boca Raton, FL},
      YEAR = {2009},
     PAGES = {xvi+267},
}

@article {CDX,
    AUTHOR = {Cantat, S. and D\'{e}serti, J. and Xie, J.},
     TITLE = {Three chapters on {C}remona groups},
   JOURNAL = {Indiana Univ. Math. J.},
  FJOURNAL = {Indiana University Mathematics Journal},
    VOLUME = {70},
      YEAR = {2021},
    NUMBER = {5},
     PAGES = {2011--2064},
}

@article {Diller_Favre,
    AUTHOR = {Diller, J. and Favre, C.},
     TITLE = {Dynamics of bimeromorphic maps of surfaces},
   JOURNAL = {Amer. J. Math.},
  FJOURNAL = {American Journal of Mathematics},
    VOLUME = {123},
      YEAR = {2001},
    NUMBER = {6},
     PAGES = {1135--1169},
     }

@incollection {Dolgachev_Iskovskikh,
    AUTHOR = {Dolgachev, I. and Iskovskikh, V.},
     TITLE = {Finite subgroups of the plane {C}remona group},
 BOOKTITLE = {Algebra, arithmetic, and geometry: in honor of {Y}u. {I}.
              {M}anin. {V}ol. {I}},
    SERIES = {Progr. Math.},
    VOLUME = {269},
     PAGES = {443--548},
 PUBLISHER = {Birkh\"{a}user Boston, Boston, MA},
      YEAR = {2009},
}

@article {Kuznetsova_Regularizations,
    AUTHOR = {Kuznetsova, A.},
     TITLE = {Regularizations of pseudo-automorphisms with positive
              algebraic entropy},
   JOURNAL = {J. Mod. Dyn.},
  FJOURNAL = {Journal of Modern Dynamics},
    VOLUME = {19},
      YEAR = {2023},
     PAGES = {237--268}
}

@article {Lonjou_Urech_CAT0,
    AUTHOR = {Lonjou, A. and Urech, C.},
     TITLE = {Actions of {C}remona groups on {${\rm CAT}(0)$} cube
              complexes},
   JOURNAL = {Duke Math. J.},
  FJOURNAL = {Duke Mathematical Journal},
    VOLUME = {170},
      YEAR = {2021},
    NUMBER = {17},
     PAGES = {3703--3743},
}

@article {TTT_dynamical_degrees,
    AUTHOR = {Truong, T. T.},
     TITLE = {Relative dynamical degrees of correspondences over a field of
              arbitrary characteristic},
   JOURNAL = {J. Reine Angew. Math.},
  FJOURNAL = {Journal f\"{u}r die Reine und Angewandte Mathematik. [Crelle's
              Journal]},
    VOLUME = {758},
      YEAR = {2020},
     PAGES = {139--182},
}

\end{document}